\documentclass{amsart}

\usepackage{amsmath, amssymb}
\usepackage{mathrsfs}
\usepackage{mathtools}
\usepackage{graphicx}
\usepackage[sort&compress, square, comma, numbers]{natbib}

\newtheorem{theorem}{Theorem}
\newtheorem{corollary}{Corollary}
\newtheorem{problem}{Problem}

\newcommand{\NN}{\mathbb{N}}

\newcommand{\RR}{\mathbb{R}}

\newcommand{\rmd}{\,\mathrm{d}}

\newcommand{\ce}{\coloneqq}
\newcommand{\ec}{\eqqcolon}

\newcommand{\pd}[2]{\frac{\partial #1}{\partial #2}}
\newcommand{\lap}{\Delta}
\newcommand{\rot}{\nabla\times}
\renewcommand{\div}{\nabla\cdot}
\newcommand{\grad}{\nabla}

\newcommand{\x}{\mathbf{x}}
\newcommand{\y}{\mathbf{y}}
\newcommand{\bnu}{\boldsymbol{\nu}}
\newcommand{\eps}{\varepsilon}
\newcommand{\bphi}{\boldsymbol{\phi}}
\newcommand{\E}{\mathbf{E}}
\renewcommand{\H}{\mathbf{H}}

\renewcommand{\P}{\mathbf{P}}
\newcommand{\G}{\mathbf{G}}
\renewcommand{\l}{ {\boldsymbol{\lambda} }}

\newcommand{\oeps}{{\bar{\eps}}}
\newcommand{\oE}{{\bar{\E}}}
\newcommand{\ol}{{\bar{\l}}}
\newcommand{\ih}{{\Pi_h}}
\newcommand{\rh}{{r_h}}
\newcommand{\maxwell}{\mathscr{D}}
\newcommand{\adjoint}{\mathscr{A}}
\renewcommand{\th}{\mathcal{T}_h}
\newcommand{\It}{\mathcal{I}_\tau}
\newcommand{\ot}{ {\Omega_T} }
\newcommand{\gt}{ {\Gamma_T} }
\newcommand{\diam}{\operatorname{diam}}
\newcommand{\vht}{V_h}
\newcommand{\uht}{U_h}

\newcommand{\abs}[1]{\left\lvert #1 \right\rvert}
\newcommand{\norm}[1]{\left\lVert{\textstyle #1 }\right\rVert}
\newcommand{\sca}[2]{\left\langle{\textstyle #1},\,{\textstyle #2}\right\rangle}
\newcommand{\jm}[2]{\left\{\textstyle{ #1 }\right\}_{\mathrm{#2}}}
\newcommand{\aj}[2]{\left[\textstyle{ #1 }\right]_{\mathrm{#2}}}

\begin{document}
\author{John Bondestam Malmberg}
\address{Department of Mathematical Sciences\\ Chalmers University of Technology and University of Gothenburg\\ SE-412 96 Gothenburg, Sweden}
\email{john.bondestam.malmberg@chalmers.se}
\date{\today}

\begin{abstract}
We present a posteriori error estimates for finite element approximations in a minimization approach to a coefficient inverse problem. The problem is that of reconstructing the dielectric permittivity $\eps = \eps(\x)$, $\x\in\Omega\subset\RR^3$, from boundary measurements of the electric field. The electric field is related to the permittivity via Maxwell's equations. The reconstruction procedure is based on minimization of a Tikhonov functional where the permittivity, the electric field and a Lagrangian multiplier function are approximated by peicewise polynomials. Our main result is an estimate for the difference between the computed coefficient $\eps_h$ and the true minimizer $\eps$, in terms of the computed functions.
\end{abstract}

\title[Error estimation in reconstruction of dielectric permittivity]{A posteriori error estimation in a finite element method for reconstruction of dielectric permittivity}
\maketitle

\section{Introduction} \label{intro}
In this note we study an adaptive finite element method for the reconstruction of a dielectric permittivity function $\eps=\eps(\x)$, $\x\in\Omega$, where $\Omega\subset\RR^3$ is a bounded domain with (piecewise) smooth boundary $\Gamma$. This is a coefficient inverse problem (CIP) for Maxwell's equations, where the dielectric permittivity function $\eps$, acting as the coefficient in the equations, characterizes an inhomogeneous, isotropic, non-magnetic, non-conductive medium in $\Omega$. Possible applications include detection of explosives in airport security and detection of land mines.

The method studied is based on minimization of a Tikhonov functional, where the functions involved are approximated by piecewise polynomials. It is intended as a second stage in a two-stage numerical procedure for the reconstruction of a dielectric permittivity. On the first stage, described in \cite{bk12, bk12b}, a good initial approximation $\eps_0$ of the dielectric permittivity function is obtained by a globally convergent method. This initial approximation is then refined on the second stage.

The version of the second stage considered here was introduced in \cite{m14}. Another version was studied theoretically and numerically in \cite{b11, bkk10, bk10a, btkm14, btkm14b}. There were two main reasons for introducing the new version of the second stage in \cite{m14}. The first reason was to handle a discrepancy between theory and implementation which was present in the previous version. This discrepancy was primarily due to the fact that the dielectric permittivity was approximated by a piecewise constant function, while the theory required higher regularity. In spite of that discrepancy, reasonable reconstructions were obtained, but it remained to be seen whether the new version of \cite{m14} could produce even more accurate reconstructions.

The second reason to introduce the version of \cite{m14} was to incorporate the divergence free condition for the electric displacement directly into the differential equation, without having to introduce an additional stabilizing penalty term as was done in \cite{b11, bkk10, bk10a, btkm14, btkm14b}.

In \cite{b11}, an a posteriori error estimate for a Lagrangian functional was derived. A similar estimate was given in \cite{m14}, but there the amount of detail provided in the proof was, for the sake of brevity, kept to a minimum. Here we give the fully detailed proof of that estimate. Moreover, we extend the error analysis also to include a posteriori error estimation for the Tikhonov functional, as well as for the permittivity function itself. The arguments which we use here could easily be adapted to obtain such estimates also for the original version of the second stage considered in \cite{b11, bkk10, bk10a, btkm14, btkm14b}.

The remaining part of this note is structured as follows: In the next section we present the mathematical formulations of the direct and inverse problems and present the basic results prior to discretization of the problems. In Section~\ref{error_analysis} we state the finite element formulations and perform the error analysis. Some concluding remarks are given in Section~\ref{conclusion}.

\section{The direct and inverse problems} \label{dir_inv}
Before proceeding with the mathematical statement of the problem, we introduce some notation. For the bounded domain $\Omega\subset \RR^3$ with boundary $\Gamma$, we write $\ot\ce \Omega\times(0,\,T)$ and $\gt\ce\Gamma\times(0,\,T)$, where $T>0$ is a (sufficiently large) fixed time. If $X\subset\RR^n$, $n\in\NN$, is a domain, we define the norm $\norm{\cdot}_{X,\,m}\ce\norm{\cdot}_{H^m(X)}$ and corresponding inner product $\sca{\cdot}{\cdot}_{X,\,m}\ce\sca{\cdot}{\cdot}_{H^m(X)}$, where $H^m(X)$ is the $L_2$-based Sobolev space of order $m$ over $X$, with respect to the usual Lebesgue measure. To simplify notation, we will drop the index $m$ whenever it is zero.

Let $V^\eps\ce H^3(\Omega)$. We define the set of admissible dielectric permittivity functions
\begin{equation} \label{admissible}
U^\eps\ce\{v\in V^\eps: 1 \leq v(\x) \leq \eps_{\mathrm{max}}~\forall \x\in\Omega,~v\rvert_\Gamma\equiv1,~\grad v\rvert_\Gamma\equiv 0\}
\end{equation}
for some known but not necessarily small upper bound $\eps_{\mathrm{\max}}$. The set $U^\eps$ is defined to describe a heterogeneous medium in $\Omega$, immersed in a constant background with permittivity 1 in $\RR^3\setminus\Omega$.

Under the assumption that $\eps\in U^\eps$ we consider Maxwell's equations for an isotropic, non-magnetic, non-conductive medium in $\Omega$:
\begin{align}
\pd{(\mu\H)}{t} + \rot\E &= 0 && \text{in } \ot, \label{mw1}\\
\pd{(\eps\E)}{t} - \rot\H &= 0 && \text{in } \ot, \label{mw2}\\
 \div(\mu\H) = \div(\eps\E) &= 0 && \text{in } \ot, \label{mw3}
\end{align}
where $\H=\H(\x,\,t)$ and $\E=\E(\x,\,t)$, $(\x,\,t)\in\ot$, denote the magnetic and electric fields, respectively, and $\mu>0$ is the constant magnetic permeability. By scaling, we may assume that $\mu=1$.

To obtain an equation involving only $\eps$ and $\E$, we combine the curl of \eqref{mw1} and derivative of \eqref{mw2} with respect to $t$ to obtain the second order equation
\begin{align*}
\eps\pd{^2\E}{t^2} + \rot(\rot\E) &= 0 && \text{in } \ot.
\end{align*}
To incorporate \eqref{mw3} we proceed as in \cite{m14} to expand $\rot(\rot\E)=-\lap\E + \grad(\div\E)$ and use
\begin{align*}
\div\E = \div\left(\frac{\eps\E}{\eps}\right) = \frac{\div(\eps\E)}{\eps} - \frac{\grad\eps\cdot\E}{\eps},
\end{align*}
where the term $\div(\eps\E)/\eps$ vanishes in view of \eqref{mw3}.

Thus, after completing with boundary and initial conditions, we obtain the system
\begin{equation} \label{maxwells}
\begin{aligned}
&\eps\pd{^2\E}{t^2}-\lap\E - \grad\left(\frac{\grad\eps\cdot\E}{\eps}\right) = 0 &&\text{in } \ot,\\
&\pd{\E}{\bnu} = \P &&\text{on } \gt, \\
&\E(\cdot,\,0) = \pd{\E}{t}(\cdot, \,0) = 0 &&\text{in } \Omega,
\end{aligned}
\end{equation}
where $\pd{}{\bnu}=\bnu\cdot\grad$, $\bnu$ denotes the outward unit normal on $\Gamma$, and $\P\in[L_2(\gt)]^3$ is given Neumann data (see Section~4 of \cite{btkm14} for details). For well-posedness of problems of this class, we refer to \cite{l85}.

The mathematical statement of the coefficient inverse problem is:

\begin{problem} \label{cip}
Given time-resolved boundary observations $\G\in[L_2(\gt)]^3$ of the electric field, determine $\eps\in U^\eps$ such that $\E = \G$ on $\gt$.
\end{problem}

The observations $\G$ represents either experimental or (partially) simulated data, see \cite{btkm14}.

Uniqueness of the solution of coefficient inverse problems of this type is typically obtained via the method of Carleman estimates \cite{bk81}. Examples where this method is applied to inverse problems for Maxwell's equations can be found in, for example, \cite{k86}, \cite{bcs12} for simultaneous reconstruction of two coefficients, and \cite{li05, ly07} for bi-isotropic and anisotropic media.
However, this technique requires non-vanishing initial conditions for the underlying partial differential equation, which is not the case here. Thus, currently, uniqueness of the solution for the problem we study is not known. For the purpose of this work, we will assume that uniqueness holds. This assumption is justified by the numerical results presented in \cite{btkm14, btkm14b}.

We introduce the space $V^\mathrm{dir}\ce\{\mathbf{v}\in[H^1(\ot)]^3:\mathbf{v}(\cdot,\,0)=0\}$ for solutions to the direct problem, and $V^\mathrm{adj}\ce\{\mathbf{v}\in[H^1(\ot)]^3:\mathbf{v}(\cdot,\,T)=0\}$ for adjoint solutions. Both spaces are equipped with the usual norm and inner product on $[H^1(\ot)]^3$.
Then, by multiplying the first equation in \eqref{maxwells} by a test function $\bphi\in V^\mathrm{adj}$ and integration over $\ot$, we obtain, after integration by parts,
\begin{equation} \label{weakmaxwell}
\begin{aligned}
0 &=
 -\sca{\eps\pd{\E}{t}}{\pd{\bphi}{t}}_\ot
 + \sca{\eps\pd{\E}{t}(\cdot,\,T)}{\bphi(\cdot,\,T)}_\Omega
 - \sca{\eps\pd{\E}{t}(\cdot,\,0)}{\bphi(\cdot,\,0)}_\Omega \\&\quad
 + \sca{\grad\E}{\grad\bphi}_\ot 
 - \sca{\pd{\E}{\bnu}}{\bphi}_\gt
 + \sca{\frac{\grad\eps\cdot\E}{\eps}}{\div\bphi}_\ot
 - \sca{\frac{\grad\eps\cdot\E}{\eps}}{\bnu\cdot\bphi}_\gt\\
&= 
 -\sca{\eps\pd{\E}{t}}{\pd{\bphi}{t}}_\ot
 + \sca{\grad\E}{\grad\bphi}_\ot 
 + \sca{\frac{\grad\eps\cdot\E}{\eps}}{\div\bphi}_\ot
 - \sca{\P}{\bphi}_\gt \\
& \ec \maxwell(\eps,\,\E,\,\bphi),
\end{aligned}
\end{equation}
where the second equality holds because $\bphi(\cdot,\,T)=0$, $\pd{\E}{t}(\cdot,\,0)=0$, $\pd{\E}{\bnu}=\P$ on $\gt$, and $\grad\eps=0$ on $\gt$. This leads to the following weak description of the electric field:
\begin{problem} \label{weakdirect}
Given $\eps\in U^\eps$, determine $\E\in V^\mathrm{dir}$ such that $\maxwell(\eps,\,\E,\,\bphi)=0$ for every $\bphi\in V^\mathrm{adj}$.
\end{problem}

Let $\E_\eps\in V^\mathrm{dir}$ denote the solution to Problem~\ref{weakdirect} for a given $\eps\in U^\eps$. We can then define the Tikhonov functional $F\colon U^\eps\to \RR_+$,
\begin{equation} \label{tikhonovfunctional}
F(\eps)=F(\eps,\,\E_\eps)\ce \frac{1}{2}\norm{(\E_\eps - \G)z_\delta}_\gt^2 + \frac{\alpha}{2}\norm{\eps - \eps_0}_\Omega^2,
\end{equation}

\begin{center}
  \begin{figure}
    \includegraphics[width=0.9\textwidth]{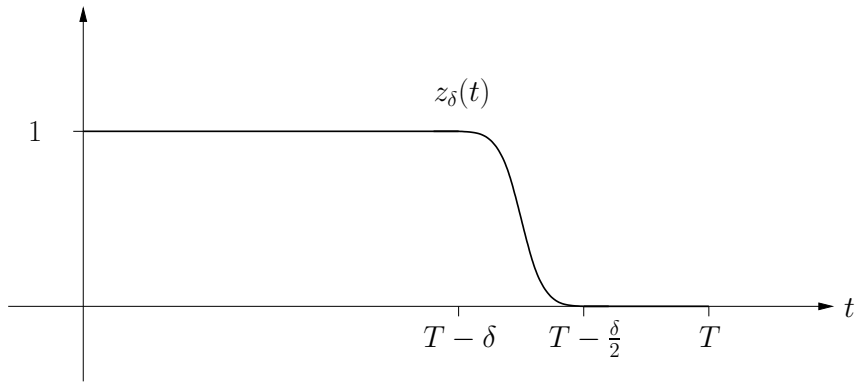}
    \caption{Schematic illustration of the cut-off function $z_\delta$ appearing in the Tikhonov functional \eqref{tikhonovfunctional}.}\label{zdelta}
  \end{figure}
\end{center}
where $\alpha>0$ is a regularization parameter and $z_\delta=z_\delta(t)\in C^\infty([0,\,T])$ is a cut-off function for the data, dropping from a constant level of 1 to a constant level of 0 within the small interval $(T-\delta,\,T-\delta/2)$, $\delta \ll T$, as schematically shown in Figure \ref{zdelta}. The function $z_\delta$ is introduced to ensure data compatibility in the adjoint problem arising in the minimization of \eqref{tikhonovfunctional}.

How to choose the regularization parameter $\alpha$ with respect to the level of noise in the data is a widely studied topic. Several methods exist, examples are the (generalized) discrepancy principle \cite{tgsy95} and iterative methods \cite{bks11}. For the results presented here, we regard $\alpha$ as a fixed parameter.

As remarked before, the initial approximation $\eps_0$ is obtained using the globally convergent method, as described in, for instance \cite{bk12b}. This means in particular that if $\eps_0$ is sufficiently close to an ideal solution $\eps^*$, corresponding to noiseless data $\G^*$, and if the regularization parameter $\alpha$ is chosen appropriately with respect to the level of noise in the data $\G$, then by Theorem~3.1 of \cite{bkk10}, the Tikhonov functional $F$ is strongly convex in a neighborhood $\mathcal{N}\subset V^\eps$ of $\eps_0$. If so, then in particular there exists a constant $c>0$ such that for every $\eps_1$, $\eps_2\in\mathcal{N}\cap U^\eps$, 
\begin{equation} \label{convexity}
c\norm{\eps_1 - \eps_2}^2_{V^\eps} \leq F'(\eps_1;\,\eps_1 - \eps_2) - F'(\eps_2;\,\eps_1 - \eps_2),
\end{equation}
where $F'(\eps;\,\oeps)$ denotes the Fr\'echet derivative of $F$ at $\eps$, acting on $\oeps$.

Throughout the remaining part of this text we will assume that the hypothesis of Theorem~3.1 of \cite{bkk10}, and hence strong convexity, holds. Then we may seek a minimizer $\eps\in U^\eps$ of $F$ by applying any gradient based method (such as steepest descent, quasi-Newton, or conjugate gradient), starting from $\eps_0$.

Such an approach requires that we compute the Fr\'echet derivative of $F$, which is complicated since it involves the implicit dependence of $\E_\eps$ upon $\eps$. To simplify the analysis, in the spirit of optimal control (see for example \cite{bkr00, kl10} for the general theory and some specific examples), we introduce the Lagrangian
\begin{equation*} \label{lagrangian}
L(u)\ce F(\eps,\,\E) + \maxwell(\eps,\,\E,\,\l),
\end{equation*}
where $u=(\eps,\,\E,\,\l)\in U \ce U^\eps\times V^\mathrm{dir}\times V^\mathrm{adj} \subset V\ce V^\eps\times V^\mathrm{dir}\times V^\mathrm{adj}$, $F(\eps,\,\E)$ was defined in \eqref{tikhonovfunctional}, and $\maxwell(\eps,\,\E,\,\l)$ was defined in \eqref{weakmaxwell}.

We can now minimize $F$ over $U^\eps$ by minimizing $L$ over $U$. With the strong convexity as above, this would imply that we solve
\begin{problem} \label{minlag}
Find $u\in U$ such that $L'(u;\,v)=0$ for every $v\in V$.
\end{problem}
Again we use the notation $L'(u;\,v)$ for the Fr\'echet derivative of $L$ at $u$, acting on $v$. It can be shown (see Proposition 1 of \cite{m14}) that
\begin{equation*}
L'(u;\,v)=\pd{L}{\eps}(u;\,\oeps) + \pd{L}{\E}(u;\,\oE) + \pd{L}{\l}(u;\,\ol),
\end{equation*}
where $u=(\eps,\,\E,\,\l)\in U$, $v=(\oeps,\,\oE,\,\ol)\in V$, and
\begin{equation}
\begin{aligned} \label{dl}
\pd{L}{\eps}(u;\,\oeps) &\ce \alpha\sca{\eps-\eps_0}{\oeps}_\Omega
 - \sca{\pd{\E}{t}\cdot\pd{\l}{t}}{\oeps}_\ot
 + \sca{(\div\l)\E}{\grad\left(\frac{\oeps}{\eps}\right)}_\ot,\\
 \pd{L}{\E}(u;\,\oE) &\ce \sca{(\E-\G)z_\delta^2}{\oE}_\gt
 - \sca{\eps\pd{\l}{t}}{\pd{\oE}{t}}_\ot
 + \sca{\grad\l}{\grad\oE}_\ot\\ &\qquad
 + \sca{\frac{\div\l}{\eps}\grad\eps}{\oE}_\ot  \ec \adjoint(\eps,\,\l,\,\oE),\\
\pd{L}{\l}(u;\,\ol) &= \maxwell(\eps,\,\E,\,\ol).
\end{aligned}
\end{equation}

In particular, we note that the solution $u=(\eps,\,\E,\,\l)$ to Problem~\ref{minlag} must satisfy $\maxwell(\eps,\,\E,\,\ol)=0$ for every $\ol\in V^\mathrm{adj}$ and $\adjoint(\eps,\,\l,\,\oE) = 0$ for every $\oE\in V^\mathrm{dir}$. The former means that $\E$ solves Problem~\ref{weakdirect} and the latter that $\l$ solves the following adjoint problem:
\begin{problem} \label{weakadjoint}
Given $\eps\in U^\eps$, determine $\l\in V^\mathrm{adj}$ such that $\adjoint(\eps,\,\l,\,\bphi)=0$ for every $\bphi\in V^\mathrm{dir}$.
\end{problem}

The functional $\adjoint$ in Problem~\ref{weakadjoint} was defined in \eqref{dl}. The problem can be seen as a weak analogue of the following system, adjoint to \eqref{maxwells}:
\begin{equation*}
\begin{aligned}
&\eps\pd{^2\l}{t^2} - \lap\l - \frac{\div\l}{\eps}\grad\eps = 0 && \text{in }  \ot,\\
&\pd{\l}{\bnu}=-(\E - \G)z_\delta^2 && \text{on } \gt,\\
&\l(\cdot,\,T) = \pd{\l}{t}(\cdot,\,T) = 0 && \text{in } \Omega. 
\end{aligned}
\end{equation*}

These observations will be used in the error analysis to be described below. But first we shall make some remarks concerning the relation between the Fr\'echet derivative of Tikhonov functional and that of the Lagrangian.

Let $u_\eps = (\eps,\,\E_\eps,\,\l_\eps)$ be the element of $U$ obtained by taking $\E_\eps$ as the solution to Problem~\ref{weakdirect} and $\l_\eps$ as the solution to Problem~\ref{weakadjoint} for the given $\eps\in U^\eps$. Then, under assumption of sufficient stability of the weak solutions $\E_\eps$ and $\l_\eps$ with respect to $\eps$, the observation that
\begin{equation*}
F(\eps) = F(\eps,\,\E_\eps) = F(\eps,\,\E_\eps) + \maxwell(\eps,\,\E_\eps,\,\l_\eps) = L(u_\eps),
\end{equation*}
(as $\maxwell(\eps,\,\E_\eps,\,\l_\eps)=0$) leads to 
\begin{equation} \label{dj2dl}
F'(\eps;\,\cdot) = \pd{L}{\eps}(u_\eps;\,\cdot).
\end{equation}
Estimate \eqref{convexity} and identity \eqref{dj2dl} will play an important role in the error analysis for the Tikhonov functional and for the coefficient.

\section{Finite element formulations and error analysis} \label{error_analysis}
In this section we will give finite element formulations for discretizing Problems~\ref{weakdirect}, \ref{minlag} and \ref{weakadjoint}. After that we will turn to the error analysis. We begin by defining finite-dimensional analogues of the spaces $V^\eps$, $V^\mathrm{dir}$, $V^\mathrm{adj}$, and $V$, as well as subsets corresponding to $U^\eps$ and $U$.

Let $\th\ce\{K\}$ be a triangulation of $\Omega$ and let $\It$ be a uniform partition of $(0,\,T)$ into subintervals $(t_k,\,t_{k+1}]$, $t_k=k\tau$, $k=0,\,\ldots,\,N_\tau$, of length $\tau=T/N_\tau$. With $\th$ we associate a mesh-function $h=h(\x)$ such that $h(\x)=\diam(K)$ for $\x\in K\in \th$. On these meshes we define\footnote{Observe that the dependence on the step size $\tau$ in time is not explicitly included in the notation for the finite-dimensional spaces. This is justified by the fact that $\tau$ should be selected with regard to $h$ in accordance with the Courant-Friedrichs-Lewy condition.}
\begin{equation*}
\begin{aligned}
V_h^\eps & \ce \{v\in V^\eps: v\rvert_K\in P^q(K)~\forall K\in \th\},\\
U_h^\eps & \ce V_h^\eps\cap U^\eps,\\
\vht^\mathrm{dir}& \ce \{v\in V^\mathrm{dir}:v\rvert_{K\times I}\in [P^1(K)]^3\times P^1(I)~\forall K\in\th~\forall I\in\It\},\\
\vht^\mathrm{adj}& \ce \{v\in V^\mathrm{adj}:v\rvert_{K\times I}\in [P^1(K)]^3\times P^1(I)~\forall K\in\th~\forall I\in\It\},\\
\vht &\ce V_h^\eps \times\vht^\mathrm{dir}\times\vht^\mathrm{adj},\\
\uht &\ce U_h^\eps \times\vht^\mathrm{dir}\times\vht^\mathrm{adj},
\end{aligned}
\end{equation*}
where $P^n(X)$ denotes the space of polynomials of degree at most $n\in\NN$ over $X$, and the degree $q$ used in the finite-dimensional analogue $V_h^\eps$ of $V^\eps$ is at least 1.

Using these spaces we can state finite element versions of Problems~\ref{weakdirect} and \ref{weakadjoint} as Problem~\ref{femdirect} and Problem~\ref{femadjoint}, respectively, as follows:

\begin{problem}\label{femdirect}
Given $\eps\in U^\eps$, determine $\E_h\in \vht^\mathrm{dir}$ such that $\maxwell(\eps,\,\E_h,\,\bphi_h)=0$ for every $\bphi_h\in \vht^\mathrm{adj}$.
\end{problem}

\begin{problem}\label{femadjoint}
Given $\eps\in U^\eps$, determine $\l_h\in \vht^\mathrm{adj}$ such that $\adjoint(\eps,\,\l_h,\,\bphi_h)=0$ for every $\bphi\in \vht^\mathrm{dir}$.
\end{problem}

The finite-dimensional analogue for Problem~\ref{minlag} is:
\begin{problem} \label{femminlag}
Find $u_h=(\eps_h,\,\E_h,\,\l_h)\in \uht$ such that $L'(u_h,\,v)=0$ for every $v\in \vht$.
\end{problem}
The same remark that was made in conjunction with Problem~\ref{minlag} is also valid here: it holds that $\E_h$ solves Problem~\ref{femdirect} and $\l_h$ solves Problem~\ref{femadjoint} for $\eps=\eps_h$.

We will now focus on estimations of the difference between the solution to Problem~\ref{minlag} and Problem~\ref{femminlag}. We begin by introducing some additional notation. For $v=(\eps,\,\E,\,\l)\in V$ we denote (with some slight abuse of notation) its interpolant in $\vht$ by 
\begin{align*}
\ih v= (\ih\eps,\,\ih\E,\,\ih\l),
\end{align*}
and the interpolation residual by 
\begin{align*}
\rh v = v-\ih v = (\rh\eps,\, \rh\E,\,\rh\l).
\end{align*}
 We will also need to consider jumps of discontinuous functions over $\th$ and $\It$. Let $K_1$, $K_2\in\th$ such that $\partial K_1\cap\partial K_2=e\neq\varnothing$. For $\x\in e$ we define
\begin{equation*}
\jm{v}{s}(\x)\ce\lim_{\y\to\x,\,\mathbf{y}\in K_1}v(\y) + \lim_{\y\to\x,\,\mathbf{y}\in K_2}v(\y),
\end{equation*}
so that in particular if $v=w\bnu$, where $w$ is piecewise constant on $\th$ and $\bnu$ is the outward unit normal, then $\jm{v}{s}=\jm{w\bnu}{s}=(w\bnu)\rvert_{K_1} + (w\bnu)\rvert_{K_2}$ is the normal jump across $e$. We extend $\jm{\cdot}{s}$ to every edge in $\th$ by defining $\jm{v}{s}(\x)=0$ for $\x\in K\cap\Gamma$, $K\in \th$. The corresponding maximal jump is defined by
\begin{equation*}
\aj{v}{s}(\x) \ce \max_{\y\in\partial K}\abs{\jm{v}{s}(\y)}, \quad \x\in K\in \th.
\end{equation*}

For jumps in time, we define
\begin{equation*}
\jm{v}{t}(t_k)\ce\begin{cases} {\displaystyle \lim_{s\to0+}}\big(v(t_k+s)-v(t_k-s)\big), & k=1,\,\ldots,\,N_\tau-1,\\
0 & k=0,\,N_\tau,\end{cases}
\end{equation*}
and
\begin{equation*}
\aj{v}{t}(t)\ce 
\max\{\abs{\jm{v}{t}(t_k)},\,\abs{\jm{v}{t}(t_{k+1})}\}
\quad t\in(t_k,\,t_{k+1}).
\end{equation*}


In the theorems and proofs to be presented, we will frequently use the symbols $\approx$ and $\lesssim$ to denote approximate equality and inequality, respectively, where higher order terms (with respect to mesh-size or errors) are neglected.

We are now ready to present the first a posteriori error estimate, an estimate for the Lagrangian. The theorem was first presented in \cite{m14}, but with only a very brief proof. We will here give the full details of the proof. Let us start by recalling the theorem:

\begin{theorem} (A posteriori error estimate for the Lagrangian.) \label{ape_lag}
Let $u=(\eps,\,\E,\,\l)\in U$ be the solution to Problem~\ref{minlag} and $u_h=(\eps_h,\,\E_h,\,\l_h)\in\uht$ be the solution to Problem~\ref{femminlag}. Then there exists a constant $C$, which does not depend on $u$, $u_h$, $h$, or $\tau$, such that 
\begin{equation*}
\begin{aligned}
\abs{L(u) - L(u_h)} &\lesssim C\left( \sca{\abs{R_\eps}}{h\abs{\aj{\pd{\eps_h}{\bnu}}{s}}}_\Omega \right.\\
&\qquad + \left. \sca{R_{\l,\,\Omega}}{\tau\abs{\aj{\pd{\E_h}{t}}{t}} + h\abs{\aj{\pd{\E_h}{\bnu}}{s}}}_\ot\right.\\
&\qquad + \left. \sca{R_{\l,\,\Gamma}}{\tau\abs{\aj{\pd{\E_h}{t}}{t}} + h\abs{\aj{\pd{\E_h}{\bnu}}{s}}}_\gt\right.\\
&\qquad + \left. \sca{R_{\E,\,\Omega}}{\tau\abs{\aj{\pd{\l_h}{t}}{t}} + h\abs{\aj{\pd{\l_h}{\bnu}}{s}}}_\ot\right.\\
&\qquad + \left.\sca{R_{\E,\,\Gamma}}{\tau\abs{\aj{\pd{\l_h}{t}}{t}} + h\abs{\aj{\pd{\l_h}{\bnu}}{s}}}_\gt\right),
\end{aligned}
\end{equation*}
where 
\begin{equation*}
\begin{aligned}
R_\eps&=\alpha(\eps_h - \eps_0)
 - \int_0^T\pd{\E_h}{t}(\cdot,\,t)\cdot\pd{\l_h}{t}(\cdot,\,t)\rmd t\\&\quad
 - \int_0^T\frac{\div\E_h(\cdot,\,t)\div\l_h(\cdot,\,t)}{\eps_h}\rmd t
 + \int_0^T\frac{\aj{(\bnu\cdot\E_h)(\div\l_h)}{s}}{h\eps_h}\rmd t,\\
R_{\l,\,\Omega} &= -\eps_h\frac{\aj{\pd{\l_h}{t}}{t}}{\tau}
 + \frac{\aj{\pd{\l}{\bnu}}{s}}{2h}
 + \frac{\div\l_h}{\eps_h}\grad\eps_h,\\
R_{\l,\,\Gamma}&= \pd{\l_h}{\bnu} + (\E_h - \G)z_\delta^2,\\
R_{\E,\,\Omega}&= -\eps_h\frac{\aj{\pd{\E_h}{t}}{t}}{\tau} 
 + \frac{\aj{\pd{\E_h}{\bnu}}{s}}{2h} 
 + \frac{\grad\eps_h\cdot\E_h}{\eps_h^2}\grad\eps_h
 - \frac{J_{\grad\eps_h}^\mathsf{T}\E_h + J_{\E_h}^\mathsf{T}\grad\eps_h}{\eps_h} \\&\quad
 + \frac{\aj{(\grad\eps_h\cdot\E_h)\bnu}{s}}{2h\eps_h},\\
R_{\E,\,\Gamma}&= \pd{\E_h}{\bnu}-\P.
\end{aligned}
\end{equation*}
Here $J_{\grad\eps_h}$ and $J_{\E_h}$ denotes the Jacobi matrices of $\grad\eps_h$ and $\E_h$, respectively, and $(\cdot)^\mathsf{T}$ denotes matrix transpose.
\end{theorem}

Note that if $\eps_h$ is piecewise linear, $J_{\grad\eps_h}\rvert_K\equiv0$ for every $K\in\th$, hence the corresponding term in $R_{\E,\,\Omega}$ vanishes in that case. 

In the following proof, and thereafter, $C$ is used to denote various constants of moderate size which do not depend on $u$, $u_h$, $h$, or $\tau$.

\begin{proof}
Using the definition of the Fr\'echet derivative we get
\begin{align*}
L(u) - L(u_h) &= L'(u_h;\,u - u_h) + o(\norm{u - u_h}_V)
\end{align*}
The split $u-u_h= (u - \ih u) + (\ih u - u_h) = \rh u + (\ih u - u_h)$ now gives 
\begin{align*}
L(u) - L(u_h) &= L'(u_h;\,\rh u + (\ih u - u_h)) + o(\norm{u - u_h}_V)\\
 &= L'(u_h;\,\rh u) + L'(u_h;\, \ih u - u_h) + o(\norm{u - u_h}_V).
\end{align*}
The second term vanishes since $\ih u-u_h\in V_h$ and $u_h$ solves Problem~\ref{femminlag}, and we neglect the remainder term $o(\norm{u - u_h}_V)$ since it is of higher order with respect to the error. We are then left with
\begin{equation*}
L(u) - L(u_h) \approx L'(u_h;\,\rh u) = \pd{L}{\eps}(u_h;\,\rh\eps) + \pd{L}{\E}(u_h;\,\rh\E) + \pd{L}{\l}(u_h;\,\rh\l),
\end{equation*}
and individual estimation of these three terms will give the stated result.

Starting with the first term, we observe that
\begin{equation*}
\begin{aligned}
\pd{L}{\eps}(u_h;\,\rh\eps) =
\alpha\sca{\eps_h - \eps_0}{\rh\eps}_\Omega
& -\sca{\pd{\E_h}{t}\cdot\pd{\l_h}{t}}{\rh\eps}_\ot \\
& + \sca{(\div\l_h)\E_h}{\grad\left(\frac{\rh\eps}{\eps_h}\right)}_\ot.
\end{aligned}
\end{equation*}

We aim at lifting all derivatives from the interpolation residuals, thus we split the inner product over $\ot$ in the last term above into the sum of inner products over $K_T\ce K\times(0,\,T)$, $K\in \th$:
\begin{equation*}
\sca{(\div\l_h)\E_h}{\grad\left(\frac{\rh\eps}{\eps_h}\right)}_\ot = \sum_{K\in\th}\sca{(\div\l_h)\E_h}{\grad\left(\frac{\rh\eps}{\eps_h}\right)}_{K_T}.
\end{equation*}

We now integrate by parts, using the notation $\partial K_T\ce \partial K\times(0,\,T)$, $\partial K'_T\ce (\partial K\setminus \Gamma)\times(0,\,T)$, $\partial K''_T\ce (\partial K\cap\Gamma)\times(0,\,T)$, $K\in\th$:
\begin{align*}
&\sum_{K\in\th}\sca{(\div\l_h)\E_h}{\grad\left(\frac{\rh\eps}{\eps_h}\right)}_{K_T}\\
&\quad=\sum_{K\in\th}\left(-\sca{\div\big((\div\l_h)\E_h\big)}{\frac{\rh\eps}{\eps_h}}_{K_T}
 + \sca{(\div\l_h)(\bnu\cdot\E_h)}{\frac{\rh\eps}{\eps_h}}_{\partial K_T}\right)\\
&\quad = \sum_{K\in\th}\left(-\sca{\frac{\grad(\div\l_h)\cdot\E_h}{\eps_h}}{\rh\eps}_{K_T}
 - \sca{\frac{(\div\l_h)(\div\E_h)}{\eps_h}}{\rh\eps}_{K_T}\right.\\
&\qquad\qquad\quad\left.+\sca{\frac{(\div\l_h)(\bnu\cdot\E_h)}{\eps_h}}{\rh\eps}_{\partial K_T'}
 + \sca{\frac{(\div\l_h)(\bnu\cdot\E_h)}{\eps_h}}{\rh\eps}_{\partial K_T''}\right)\\
&\quad = -\sum_{K\in\th}\sca{\frac{\grad(\div\l_h)\cdot\E_h}{\eps_h}}{\rh\eps}_{K_T}
 - \sca{\frac{(\div\l_h)(\div\E_h)}{\eps_h}}{\rh\eps}_\ot \\
&\qquad + \sum_{K\in\th}\sca{\frac{(\div\l_h)(\bnu\cdot\E_h)}{\eps_h}}{\rh\eps}_{\partial K_T'}
 + \sca{\frac{(\div\l_h)(\bnu\cdot\E_h)}{\eps_h}}{\rh\eps}_\gt.
\end{align*}

We observe that $\grad(\div\l_h)\equiv0$ on every $K_T$, $K\in\th$, since $\l_h$ is piecewise linear, and that $\eps\equiv1$ on $\Gamma$ so that $\rh\eps\rvert_\Gamma\equiv0$. With this in mind, the above calculations yields
\begin{align*}
\sca{(\div\l_h)\E_h}{\grad\left(\frac{\rh\eps}{\eps_h}\right)}_\ot &=
 - \sca{\frac{(\div\l_h)(\div\E_h)}{\eps_h}}{\rh\eps}_\ot \\&\quad
 + \sum_{K\in\th}\sca{\frac{(\div\l_h)(\bnu\cdot\E_h)}{\eps_h}}{\rh\eps}_{\partial K_T'}.
\end{align*}

In order to obtain a residual defined in the whole of $\Omega$, as opposed to one containing terms defined only on edges of elements $K\in\th$, we should manipulate the last term in the above expression further. Observe that
\begin{align*}
\sum_{K\in\th}\sca{\frac{1}{\eps_h}(\div\l_h)(\bnu\cdot\E_h)}{\rh\eps}_{\partial K_T'} = \frac{1}{2}\sum_{K\in\th}\sca{\frac{1}{\eps_h}\jm{(\div\l_h)(\bnu\cdot\E_h)}{s}}{\rh\eps}_{\partial K_T'},
\end{align*}
where the factor $\frac{1}{2}$ appears since every internal edge is counted exactly twice in the sum over all elements $K\in \th$.

Using the approximation
\begin{equation*}
\int_{\partial K}f\rmd S \approx \int_K\frac{\tilde{f}}{h_K}\rmd\x
\end{equation*}
where $\tilde{f}$ denotes the maximum of $f$ over $\partial K$ (see for instance \cite{eehj96}), we finally get
\begin{align*}
\sum_{K\in\th}\sca{\frac{1}{\eps_h}\jm{(\div\l_h)(\bnu\cdot\E_h)}{s}}{\rh\eps}_{\partial K_T'}&\approx
\sum_{K\in\th}\sca{\frac{1}{h_K\eps_h}\aj{(\div\l_h)(\bnu\cdot\E_h)}{s}}{\rh\eps}_{K_T}\\
&= \sca{\frac{1}{h\eps_h}\aj{(\div\l_h)(\bnu\cdot\E_h)}{s}}{\rh\eps}_\ot,
\end{align*} 
which gives
\begin{align*}
\pd{L}{\eps}(u_h;\,\rh\eps) &=
\alpha\sca{\eps_h - \eps_0}{\rh\eps}_\Omega
 + \sca{\pd{\E_h}{t}\cdot\pd{\l_h}{t}}{\rh\eps}_\ot \\&\qquad
 - \sca{\frac{(\div\l_h)(\div\E_h)}{\eps_h}}{\rh\eps}_\ot
 + \sca{\frac{1}{h\eps_h}\aj{(\div\l_h)(\bnu\cdot\E_h)}{s}}{\rh\eps}_\ot \\&
 = \sca{R_\eps}{\rh\eps}_\Omega.
\end{align*}

We can now estimate $\rh\eps$ in terms of $\eps_h$, using standard interpolation techniques (see for instance \cite{js95}), as 
\begin{equation*}
\abs{\rh\eps}\leq Ch^2\abs{D^2\eps}\approx Ch^2\abs{\frac{\aj{\pd{\eps_h}{\bnu}}{s}}{h}}=Ch\abs{\aj{\pd{\eps_h}{\bnu}}{s}},
\end{equation*}
where $D^2$ denotes derivatives of second order with respect to $\x$. Thus
\begin{equation*}
\abs{\pd{L}{\eps}(u_h;\,\rh\eps)}\lesssim C\sca{\abs{R_\eps}}{h\abs{\aj{\pd{\eps_h}{\bnu}}{s}}}_\Omega.
\end{equation*}

We continue with
\begin{align*}
\pd{L}{\E}(u_h;\,\rh\E) &= \sca{(\E_h - \G)z_\delta^2}{\rh\E}_\gt
 - \sca{\eps_h\pd{\l_h}{t}}{\pd{\rh\E}{t}}_\ot \\&\qquad
 + \sca{\grad\l_h}{\grad\rh\E}_\ot 
 + \sca{\frac{\div\l_h}{\eps_h}\grad\eps_h}{\rh\E}_\ot.
\end{align*}
Again, we seek to lift derivatives from the interpolation residuals, and so we use integration by parts to get
\begin{align*}
\sca{\eps_h\pd{\l_h}{t}}{\pd{\rh\E}{t}}_\ot &= \sum_{k=1}^{N_\tau}\int_{t_{k-1}}^{t_k}\sca{\eps_h\pd{\l_h}{t}}{\pd{\rh\E}{t}}_\Omega\rmd t \\
&= -\sum_{k=1}^{N_\tau}\int_{t_{k-1}}^{t_k}\sca{\eps_h\pd{^2\l_h}{t^2}}{\rh\E}_\Omega\rmd t \\&\quad
 + \sum_{k=1}^{N_\tau}\Big(\sca{\eps_h\pd{\l_h}{t}}{\rh\E}_\Omega\rvert_{t=t_k} - \sca{\eps_h\pd{\l_h}{t}}{\rh\E}_\Omega\rvert_{t=t_{k-1}}\Big)\\
&= \sum_{k = 1}^{N_\tau-1}\sca{\eps_h\jm{\pd{\l_h}{t}}{t}}{\rh\E}_\Omega\rvert_{t=t_k} 
 + \sca{\eps_h\pd{\l_h}{t}}{\rh\E}_\Omega\rvert_{t=T} \\
&\quad - \sca{\eps_h\pd{\l_h}{t}}{\rh\E}_\Omega\rvert_{t=0},
\end{align*}
where we have used the fact that $\pd{^2\l_h}{t^2}\equiv0$ on each subinterval $(t_{k-1},\,t_k)$, for the piecewise linear function $\l_h$.

Since $\rh\E(\cdot,\,0)=\pd{\l_h}{t}(\cdot,\,T)=0$, this leaves us with
\begin{equation*}
\sca{\eps_h\pd{\l_h}{t}}{\pd{\rh\E}{t}}_\ot = \sum_{k = 1}^{N_\tau-1}\sca{\eps_h\jm{\pd{\l_h}{t}}{t}}{\rh\E}_\Omega\rvert_{t=t_k}.
\end{equation*}
We now approximate the boundary terms by terms defined on the whole interval, using
\begin{equation*}
f(t_k)\approx \frac{1}{\tau}\int_{t_{k-1}}^{t_k}f(t)\rmd t,
\end{equation*}
that is
\begin{align*}
\sum_{k = 1}^{N_\tau-1}\sca{\eps_h\jm{\pd{\l_h}{t}}{t}}{\rh\E}_\Omega\rvert_{t=t_k}&\approx\sum_{k = 1}^{N_\tau-1}\frac{1}{\tau}\int_{t_{k-1}}^{t_k}\sca{\eps_h\aj{\pd{\l_h}{t}}{t}}{\rh\E}_\Omega\rmd t\\
&=\sca{\frac{\eps_h}{\tau}\aj{\pd{\l_h}{t}}{t}}{\rh\E}_\ot.
\end{align*}

Moving on to
\begin{equation*}
\sca{\grad\l_h}{\grad\rh\E}_\ot = \sum_{K\in\th}\sca{\grad\l_h}{\grad\rh\E}_{K_T},
\end{equation*}
we integrate by parts and use the fact that $\lap\l_h\equiv0$ in every $K\in\th$ to obtain
\begin{align*}
\sca{\grad\l_h}{\grad\rh\E}_\ot &=\sum_{K\in\th}\left(-\sca{\lap\l_h}{\rh\E}_{K_T} + \sca{\pd{\l_h}{\bnu}}{\rh\E}_{\partial K_T}\right)\\
&=\sum_{K\in\th}\left(\sca{\pd{\l_h}{\bnu}}{\rh\E}_{\partial K'_T} + \sca{\pd{\l_h}{\bnu}}{\rh\E}_{\partial K''_T}\right)\\
&=\frac{1}{2}\sum_{K\in\th}\sca{\jm{\pd{\l_h}{\bnu}}{s}}{\rh\E}_{\partial K'_T} + \sca{\pd{\l_h}{\bnu}}{\rh\E}_\gt.
\end{align*}

We again approximate inner products over $\partial K'_T$ by inner products over $K_T$, so that
\begin{equation*}
\sca{\grad\l_h}{\grad\rh\E}_\ot \approx \sca{\frac{1}{2h}\aj{ \pd{\l_h}{\bnu}}{s}}{\rh\E}_\ot + \sca{\pd{\l_h}{\bnu}}{\rh\E}_\gt.
\end{equation*}
Together with previous calculations, this gives
\begin{align*}
\pd{L}{\E}(u_h;\,\rh\E) &\approx  \sca{(\E_h - \G)z_\delta^2}{\rh\E}_\gt
 - \sca{\frac{\eps_h}{\tau}\aj{\pd{\l_h}{t}}{t}}{\rh\E}_\ot \\&\qquad
 + \sca{\frac{1}{2h}\aj{\pd{\l_h}{\bnu}}{s}}{\rh\E}_\ot
 + \sca{\pd{\l_h}{\bnu}}{\rh\E}_\gt \\&\qquad
 + \sca{\frac{\div\l_h}{\eps_h}\grad\eps_h}{\rh\E}_\ot\\
&=\sca{R_{\l,\,\Omega}}{\rh\E}_\ot + \sca{R_{\l,\,\Gamma}}{\rh\E}_\gt.
\end{align*}
We once more use interpolation estimates
\begin{equation*}
\abs{\rh\E}\leq C\left(h^2\abs{D^2\E} + \tau^2\abs{\pd{^2\E}{t^2}} \right)\approx C\left(h\abs{\aj{\pd{\E_h}{\bnu}}{t}} + \tau\abs{\aj{\pd{\E_h}{t}}{s}} \right).
\end{equation*}
to get
\begin{align*}
\abs{\pd{L}{\E}(u_h;\,\rh\E)} &\lesssim C\left(\sca{\abs{R_{\l,\,\Omega}}}{h\abs{\aj{\pd{\E_h}{\bnu}}{t}} + \tau\abs{\aj{\pd{\E_h}{t}}{s}} }_\ot\right. \\&\qquad
 + \left. \sca{\abs{R_{\l,\,\Gamma}}}{h\abs{\aj{\pd{\E_h}{\bnu}}{t}} + \tau\abs{\aj{\pd{\E_h}{t}}{s}}}_\gt\right).
\end{align*}

It remains to estimate
\begin{align*}
\pd{L}{\l}(u_h;\,\rh\l) &=
 - \sca{\eps_h\pd{\E_h}{t}}{\pd{\rh\l}{t}}_\ot
 + \sca{\grad\E_h}{\grad\rh\l}_\ot \\&\quad
 + \sca{\frac{\grad\eps_h\cdot\E_h}{\eps_h}}{\div\rh\l}_\ot
 - \sca{\P}{\rh\l}_\gt.
\end{align*}
Just as before, we obtain
\begin{equation*}
\sca{\eps_h\pd{\E_h}{t}}{\pd{\rh\l}{t}}_\ot \approx \sca{\frac{\eps_h}{\tau}\aj{\pd{\E_h}{t}}{t}}{\rh\l}_\ot
\end{equation*}
and
\begin{equation*}
\sca{\grad\E_h}{\grad\rh\l}_\ot\approx \sca{\frac{1}{2h} \aj{ \pd{\E_h}{\bnu} }{s} }{\rh\l}_\ot + \sca{\pd{\E_h}{\bnu}}{\rh\l}_\gt.
\end{equation*}
Consider the term
\begin{equation*}
\sca{\frac{\grad\eps_h\cdot\E_h}{\eps_h}}{\div\rh\l}_\ot = \sum_{K\in\th}\sca{\frac{\grad\eps_h\cdot\E_h}{\eps_h}}{\div\rh\l}_{K_T}.
\end{equation*}
Integration by parts yields
\begin{align*}
&\sca{\frac{\grad\eps_h\cdot\E_h}{\eps_h}}{\div\rh\l}_\ot \\&\qquad=
\sum_{K\in\th}\left(-\sca{\grad\left(\frac{\grad\eps_h\cdot\E_h}{\eps_h}\right)}{\rh\l}_{K_T}
  + \sca{\frac{\grad\eps_h\cdot\E_h}{\eps_h}\bnu}{\rh\l}_{\partial K_T}\right) \\&\qquad=
\sum_{K\in\th}\sca{\frac{\grad\eps_h\cdot\E_h}{\eps_h^2}\grad\eps_h - \frac{J^\mathsf{T}_{\grad\eps_h}\E_h + J^\mathsf{T}_{\E_h}\grad\eps_h}{\eps_h} }{\rh\l}_{K_T}\\&\qquad\quad
 + \sum_{K\in\th}\left(\sca{\frac{\grad\eps_h\cdot\E_h}{\eps_h}\bnu}{\rh\l}_{\partial K'_T}
 + \sca{\frac{\grad\eps_h\cdot\E_h}{\eps_h}\bnu}{\rh\l}_{\partial K''_T}\right) \\&\qquad=
\sca{\frac{\grad\eps_h\cdot\E_h}{\eps_h^2}\grad\eps_h - \frac{J^\mathsf{T}_{\grad\eps_h}\E_h + J^\mathsf{T}_{\E_h}\grad\eps_h}{\eps_h} }{\rh\l}_\ot \\&\qquad\quad
+ \sum_{K\in\th}\sca{\frac{1}{\eps_h}\jm{(\grad\eps_h\cdot\E_h)\bnu}{s}}{\rh\l}_{\partial K'_T}
+ \sca{\frac{\grad\eps_h\cdot\E_h}{\eps_h}\bnu}{\rh\l}_\gt,
\end{align*}
where for the second equality we have used the identity $\grad(\grad\eps_h\cdot\E_h)=J^\mathsf{T}_{\grad\eps_h}\E_h + J^\mathsf{T}_{\E_h}\grad\eps_h$.

Noting that $\grad\eps_h\rvert_\Gamma\equiv0$ as $\eps_h\in U^\eps_h$, and using the usual approximation for $\jm{\cdot}{s}$ inside elements $K\in\th$ we get
\begin{align*}
\sca{\frac{\grad\eps_h\cdot\E_h}{\eps_h}}{\div\rh\l}_\ot &\approx
\sca{\frac{\grad\eps_h\cdot\E_h}{\eps_h^2}\grad\eps_h - \frac{J^\mathsf{T}_{\grad\eps_h}\E_h + J^\mathsf{T}_{\E_h}\grad\eps_h}{\eps_h}}{\rh\l}\\&\quad + \sca{\frac{1}{2h\eps_h}\aj{(\grad\eps_h\cdot\E_h)\bnu}{s}}{\rh\l}_\ot.
\end{align*}

Combining the results for $\pd{L}{\l}(u_h;\,\rh\l)$ and estimating $\rh\l$ in terms of $\l_h$ just as $\rh\E$ was estimated in terms of $\E_h$ gives
\begin{align*}
\abs{\pd{L}{\l}(u_h;\,\rh\l)}
&\lesssim C\left(\sca{\abs{R_{\E,\,\Omega}}}{h\abs{\aj{\pd{\l_h}{\bnu}}{s}}  + \tau\abs{\aj{\pd{\l_h}{t}}{t}}}_\ot \right.\\& \qquad + \left.\sca{\abs{R_{\E,\,\Gamma}}}{h\abs{\aj{\pd{\l_h}{\bnu}}{s}} + \tau\abs{\aj{\pd{\l_h}{t}}{t}}}_\gt\right),
\end{align*}
which completes the proof.
\end{proof}

One should note that the terms in the error estimate of Theorem~\ref{ape_lag} which are derived from $\pd{L}{\l}(u_h;\,\rh\l)$ and $\pd{L}{\E}(u_h;\,\rh\E)$ estimate how accurately the solutions of Problem~\ref{weakdirect} and Problem~\ref{weakadjoint} are approximated by the solutions of Problem~\ref{femdirect} and Problem~\ref{femadjoint}, respectively, for the approximate coefficient $\eps_h$. The remaining term, $\sca{R_\eps}{h\abs{\aj{\pd{\eps_h}{\bnu}}{s}}}_\Omega$ can be interpreted as the error induced by approximating $\eps$ by $\eps_h$. Thus, if we are mainly interested in that error, or if we can postulate that the finite element approximations $\E_h$ and $\l_h$ are computed with relatively high accuracy, then $\abs{R_\eps}$ may be used as an error indicator by itself. The significance of $R_\eps$ will be further illustrated by the error estimates for the coefficient and for the Tikhonov functional.

We now proceed with an error estimate for the coefficient itself. An error estimate for the Tikhonov functional will follow as a corollary.

\begin{theorem} (A posteriori error estimate for the coefficient.) \label{ape_norm}
Suppose that the initial approximation $\eps_0$ and the regularization parameter $\alpha$ are such that the strong convexity estimate \eqref{convexity} holds. Let $u=(\eps,\,\E,\,\l)\in U$ be the solution to Problem~\ref{minlag}, and let $u_h=(\eps_h,\,\E_h,\,\l_h)\in \uht$ be the solution to Problem~\ref{femminlag}, computed on meshes $\th$ and $\It$. Denote by $\tilde\E$ and $\tilde\l$ the solutions to Problem~\ref{weakdirect} and Problem~\ref{weakadjoint}, respectively, with permittivity $\eps_h$, and set $\tilde{u} =(\eps_h,\,\tilde\E,\,\tilde\l)\in U$. Then there exists a constant $C$, which does not depend on $u$, $u_h$, $h$, or $\tau$, such that 
\begin{equation*}
\norm{\eps - \eps_h}_{V^\eps} \lesssim C( c_\eps\eta + \norm{R_\eps}_\Omega),
\end{equation*}
where $c_\eps\ce\max\{1,\,\norm{\grad\eps_h}_{L_\infty(\Omega)}\}$ and $\eta =  \eta(u_h)$ is defined by
\begin{align*}
\eta &\ce
   \sca{ \frac{1}{\tau} \abs{\aj{\pd{\l_h}{t}}{t}} + \abs{\div\l_h} }{h\abs{\aj{\pd{\E_h}{\bnu}}{s}} + \tau\abs{\aj{\pd{\E_h}{t}}{t}}  }_\ot \\ &\qquad
 + \sca{\frac{1}{\tau}\abs{\aj{\pd{\E_h}{t}}{t}}}{h\abs{\aj{\pd{\l_h}{\bnu}}{s}} + \tau\abs{\aj{\pd{\l_h}{t}}{t}} }_\ot\\&\qquad
 + \sca{\abs{\E_h}}{ \abs{\aj{\pd{\l_h}{\bnu}}{s}} + \tau\abs{\aj{\pd{\div\l_h}{t}}{t}} }_\ot.
\end{align*}
\end{theorem}

\begin{proof}
Using strong convexity \eqref{convexity}, we obtain
\begin{equation*}
\norm{\eps - \eps_h}_{V^\eps}^2\leq c \left(F'(\eps;\,\eps - \eps_h) - F'(\eps_h;\,\eps-\eps_h)\right).
\end{equation*}
Since $\eps$ minimizes $F(\eps)$ we have $F'(\eps;\eps-\eps_h)=0$ and thus
\begin{equation} \label{pfthm2_00}
\norm{\eps-\eps_h}_{V^\eps}^2\leq c\abs{F'(\eps_h;\,\eps - \eps_h)} = c\abs{\pd{L}{\eps}(\tilde u;\,\eps - \eps_h)}, 
\end{equation}
where the last equality follows from \eqref{dj2dl}.

We expand
\begin{equation} \label{pfthm2_0}
\begin{aligned}
\abs{\pd{L}{\eps}(\tilde u;\,\eps - \eps_h)} &= \abs{\pd{L}{\eps}(\tilde u;\,\eps - \eps_h) -\pd{L}{\eps}(u_h;\,\eps - \eps_h) + \pd{L}{\eps}(u_h;\,\eps - \eps_h)}\\
&\leq \abs{\pd{L}{\eps}(\tilde u;\,\eps - \eps_h) -\pd{L}{\eps}(u_h;\,\eps - \eps_h)} + \abs{\pd{L}{\eps}(u_h;\,\eps - \eps_h)}\\
&\ec \abs{\Theta_1} + \abs{\Theta_2},
\end{aligned}
\end{equation}
and estimate the two terms $\abs{\Theta_1}$ and $\abs{\Theta_2}$ separately.

For $\Theta_1$ we use the linearization
\begin{align*}
\Theta_1&=
 \pd{^2L}{\E\partial\eps}(u_h;\,\tilde\E - \E_h;\eps - \eps_h) + o(\lVert\tilde\E - \E_h\rVert_{\ot,\,1}) \\ &\quad
 + \pd{^2L}{\l\partial\eps}(u_h;\,\tilde\l - \l_h;\eps - \eps_h) + o(\lVert\tilde\l - \l_h\rVert_{\ot,\,1}),
\end{align*}
where $\pd{^2L}{\E\partial\eps}$ and $\pd{^2L}{\l\partial\eps}$ denote mixed second partial Fr\'echet derivatives of $L$. Again, the remainder terms are neglected as they are of higher order with respect to the error. Thus, after exchanging the order of differentiation, we are left with
\begin{equation}
\begin{aligned}\label{pfthm2_1}
\Theta_1 &\approx D_1\rvert_{\eps - \eps_h}\left(\pd{L}{\E}(u_h;\,\tilde\E - \E_h) + \pd{L}{\l}(u_h;\,\tilde\l - \l_h)\right),
\end{aligned}
\end{equation}
where $D_1\rvert_{\eps - \eps_h}$ denotes differentiation with respect to the first component in $u_h$ and action on $\eps - \eps_h$.

We split $\tilde\E - \E_h=(\tilde\E - \ih\tilde\E) + (\ih\tilde\E - \E_h) = \rh\tilde\E + (\ih\tilde\E - \E_h)$ and use the fact that $\l_h$ solves Problem~\ref{weakadjoint} with coefficient $\eps_h$, so that $\pd{L}{\E}(u_h;\,\ih\tilde\E-\E_h) = 0$ as $\ih\tilde\E-\E_h\in V_h^\mathrm{dir}$. This gives
\begin{align} \label{pfthm2_2}
\pd{L}{\E}(u_h;\,\tilde\E - \E_h) = \pd{L}{\E}(u_h;\,\rh\tilde\E) + \pd{L}{\E}(u_h;\,\ih\tilde\E - \E_h) = \pd{L}{\E}(u_h;\,\rh\tilde\E).
\end{align}
Similarly, we have
\begin{align} \label{pfthm2_3}
\pd{L}{\l}(u_h;\,\tilde\l - \l_h) = \pd{L}{\l}(u_h;\,\rh\tilde\l) + \pd{L}{\l}(u_h;\,\ih\tilde\l - \l_h) = \pd{L}{\l}(u_h;\,\rh\tilde\l).
\end{align}
as $\E_h$ solves Problem~\ref{weakdirect} with coefficient $\eps_h$.

Combining \eqref{pfthm2_1}, \eqref{pfthm2_2}, and \eqref{pfthm2_3} gives
\begin{align*}
\Theta_1 &\approx D_1\rvert_{\eps - \eps_h}\left(\pd{L}{\E}(u_h;\,\rh\tilde\E) + \pd{L}{\l}(u_h;\,\rh\tilde\l)\right)\\
 &= -\sca{(\eps - \eps_h)\pd{\rh\tilde\E}{t}}{\pd{\l_h}{t}}_\ot + \sca{\grad\left(\frac{\eps - \eps_h}{\eps_h}\right)\rh\tilde\E}{\div\l_h}_\ot \\
 & \quad-\sca{(\eps - \eps_h)\pd{\E_h}{t}}{\pd{\rh\tilde\l}{t}}_\ot + \sca{\grad\left(\frac{\eps - \eps_h}{\eps_h}\right)\E_h}{\div\rh\tilde\l}_\ot.
\end{align*}

In the same manner as in the proof of Theorem~\ref{ape_lag}, we integrate by parts in time and approximate jumps to get
\begin{align*}
-\sca{(\eps - \eps_h)\pd{\rh\tilde\E}{t}}{\pd{\l_h}{t}}_\ot &= \sum_{k=1}^{N_\tau}\int_{t_{k-1}}^{t_k}\sca{(\eps - \eps_h)\rh\tilde\E}{\jm{\pd{\l_h}{t}}{t}}_\Omega\rmd t \\
&\approx \sca{(\eps - \eps_h)\rh\tilde\E}{\frac{1}{\tau}\aj{\pd{\l_h}{t}}{t}}_\ot
\end{align*}
and
\begin{align*}
-\sca{(\eps - \eps_h)\pd{\E_h}{t}}{\pd{\rh\tilde\l}{t}}_\ot &= \sum_{k=1}^{N_\tau}\int_{t_{k-1}}^{t_k}\sca{ (\eps - \eps_h)\jm{ \pd{\E_h}{t} }{t} }{ \rh\tilde\l }_\Omega\rmd t\\ 
&\approx \sca{ (\eps - \eps_h)\frac{1}{\tau}\aj{ \pd{\E_h}{t} }{t} }{ \rh\tilde\l }_\ot.
\end{align*}
Thus
\begin{align*}
  \Theta_1 &\lesssim \sca{\abs{\eps - \eps_h}\frac{1}{\tau}\abs{\aj{\pd{\l_h}{t}}{t}}}{ \abs{\rh\tilde\E}}_\ot
 + \sca{\abs{\grad\left(\frac{\eps - \eps_h}{\eps_h}\right)}\abs{\div\l_h} }{ \abs{\rh\tilde\E}}_\ot\\
 & \quad+\sca{\abs{\eps - \eps_h}\frac{1}{\tau}\abs{\aj{\pd{\E_h}{t}}{t}}}{\abs{\rh\tilde\l}}_\ot
 + \sca{\abs{\grad\left(\frac{\eps - \eps_h}{\eps_h}\right)}\abs{\E_h}}{\abs{\div\rh\tilde\l}}_\ot\\
&\leq \norm{\eps - \eps_h}_{L_\infty(\Omega)}\left(\sca{\frac{1}{\tau}\abs{\aj{\pd{\l_h}{t}}{t}}}{ \abs{\rh\tilde\E}}_\ot + \sca{\frac{1}{\tau}\abs{\aj{\pd{\E_h}{t}}{t}}}{\abs{\rh\tilde\l}}_\ot\right) \\&\quad
+ \norm{\grad\left(\frac{\eps - \eps_h}{\eps_h}\right)}_{L_\infty(\Omega)}\left(\sca{\abs{\div\l_h} }{ \abs{\rh\tilde\E}}_\ot + \sca{\abs{\E_h}}{\abs{\div\rh\tilde\l}}_\ot\right).
\end{align*}

Note that
\begin{align*}
\norm{\grad\left(\frac{\eps - \eps_h}{\eps_h}\right)}_{L_\infty(\Omega)} &=
 \norm{\frac{\grad(\eps-\eps_h)}{\eps_h}  - \frac{(\eps-\eps_h)\grad\eps_h}{\eps_h^2}}_{L_\infty(\Omega)} \\
&\leq\norm{\frac{1}{\eps_h}}_{L_\infty(\Omega)}\norm{\grad(\eps - \eps_h)}_{L_\infty(\Omega)} \\&\quad + \norm{\frac{1}{\eps_h^2}}_{L_\infty(\Omega)}\norm{\grad\eps_h}_{L_\infty(\Omega)}\norm{\eps-\eps_h}_{L_\infty(\Omega)}
\end{align*}
and observe following facts:
\begin{align*}
\norm{\eps-\eps_h}_{L_\infty(\Omega)} + \norm{\grad(\eps-\eps_h)}_{L_\infty(\Omega)} &\leq C\norm{\eps-\eps_h}_{V^\eps},\\
\norm{\frac{1}{\eps_h^p} }_{L_\infty(\Omega)} &\leq 1, \quad p \geq 0,
\end{align*}
the first following from the Sobolev inequality and the second from noting that $1\leq\eps_h(\x)\leq \eps_\mathrm{max}$, $\x\in\Omega$, by \eqref{admissible}.  

Using these observations, and interpolation estimates
\begin{align*}
\abs{\rh\tilde\E} &\leq C\left(h\abs{ \aj{ \pd{\E_h}{\bnu} }{ s } } + \tau\abs{\aj{\pd{\E_h}{t}}{t}}\right),\\
\abs{\rh\tilde\l} &\leq C\left(h\abs{ \aj{ \pd{\l_h}{\bnu} }{ s } } + \tau\abs{\aj{\pd{\l_h}{t}}{t}}\right),\\
\abs{\div\rh\tilde\l} &\leq C\left(\abs{ \aj{ \pd{\l_h}{\bnu} }{ s } } + \tau\abs{\aj{\pd{\div\l_h}{t}}{t}}\right),
\end{align*}
we get
\begin{equation}\label{pfthm2_4}
\begin{aligned}
 \Theta_ 1 &\lesssim C\left(\sca{\frac{1}{\tau}\abs{\aj{\pd{\l_h}{t}}{t}}}{h\abs{ \aj{ \pd{\E_h}{\bnu} }{ s } } + \tau\abs{\aj{\pd{\E_h}{t}}{t}} }_\ot\right. \\&\qquad
 + \left. \norm{\grad\eps_h}_{L_\infty(\Omega)}\sca{\abs{\div\l_h} }{ h\abs{ \aj{ \pd{\E_h}{\bnu} }{ s } } + \tau\abs{\aj{\pd{\E_h}{t}}{t}} }_\ot \right.\\&\qquad
 + \left. \sca{\frac{1}{\tau}\abs{\aj{\pd{\E_h}{t}}{t}}}{ h\abs{ \aj{ \pd{\l_h}{\bnu} }{ s } } + \tau\abs{\aj{\pd{\l_h}{t}}{t}} }_\ot\right.\\&\qquad + \left. \norm{\grad\eps_h}_{L_\infty(\Omega)}\sca{\abs{\E_h}}{\abs{ \aj{ \pd{\l_h}{\bnu} }{ s } } + \tau\abs{\aj{\pd{\div\l_h}{t}}{t}}}_\ot \right)\norm{\eps - \eps_h}_{V^\eps} \\&
\leq Cc_\eps\eta\norm{\eps - \eps_h}_{V^\eps},
\end{aligned}
\end{equation}
where $c_\eps$ and $\eta$ were defined in the statement of the theorem.

Turning to $\Theta_2$ of \eqref{pfthm2_0}, we use the techniques of the proof of Theorem~\ref{ape_lag} to estimate
\begin{equation} \label{pfthm2_5}
\abs{\Theta_2} \lesssim C\sca{\abs{R_\eps}}{\abs{\eps-\eps_h}}_\Omega \leq C\norm{R_\eps}_\Omega\norm{\eps - \eps_h}_\Omega \leq C\norm{R_\eps}_\Omega\norm{\eps - \eps_h}_{V^\eps}.
\end{equation}

Combining estimates \eqref{pfthm2_4} and \eqref{pfthm2_5} with \eqref{pfthm2_00} and \eqref{pfthm2_0}, we conclude that
\begin{equation*}
\norm{\eps - \eps_h}_{V^\eps}^2 \lesssim C\left(c_\eps\eta\norm{\eps - \eps_h}_{V^\eps} + \norm{R_\eps}_\Omega\norm{\eps - \eps_h}_{V^\eps}\right),
\end{equation*}
and the result follows.
\end{proof}

Again, just as for the error estimate for the Lagrangian, we see that if the numerical errors for solving the direct and adjoint problems are relatively small, that is, when $\tilde u\approx u_h$ with relatively high accuracy, then $\norm{R_\eps}_\Omega$ dominates the error estimate.

\begin{corollary}(A posteriori error estimate for the Tikhonov functional.) \label{ape_tikh}
Under the hypothesis of Theorem~\ref{ape_norm}, we have
\begin{equation*}
\abs{F(\eps) - F(\eps_h)} \lesssim C\left(c_\eps^2\eta^2 + \norm{R_\eps}_\Omega^2\right),
\end{equation*}
with $c_\eps$ and $\eta$ as defined in Theorem~\ref{ape_norm}.
\end{corollary}

\begin{proof}
Using the definition of the Fr\'echet derivative and \eqref{dj2dl} we get
\begin{align*}
F(\eps) - F(\eps_h) &= F'(\eps_h;\eps - \eps_h) + o(\norm{\eps - \eps_h}_{V^\eps})\\
                   &= \pd{L}{\eps}(\tilde u;\,\eps - \eps_h) + o(\norm{\eps - \eps_h}_{V^\eps}).
\end{align*}
Neglecting the remainder term as it is of higher order with respect to the error, and estimating $\pd{L}{\eps}(\tilde u;\,\eps - \eps_h)$ as in the proof of Theorem~\ref{ape_norm}, we obtain
\begin{align*}
\abs{F(\eps) - F(\eps_h)} \lesssim C(c_\eps\eta + \norm{R_\eps}_{V^\eps})\norm{\eps - \eps_h}_{V^\eps}.
\end{align*}
Applying Theorem~\ref{ape_norm} to estimate $\norm{\eps - \eps_h}_{V^\eps}$, we arrive at
\begin{align*}
\abs{F(\eps) - F(\eps_h)} \lesssim C(c_\eps\eta + \norm{R_\eps}_{V^\eps})^2\leq C\left(c_\eps^2\eta^2 + \norm{R_\eps}_{V^\eps}^2\right).
\end{align*}
\end{proof}

\section{Conclusion} \label{conclusion}
We have presented three a posteriori error estimates for an adaptive finite element method for the coefficient inverse problem, Problem \ref{cip}: for the Lagrangian, for the Tikhonov functional and for the coefficient. The latter two are presented here for the first time. Each estimator consists essentially of three parts, an estimate for the error resulting from finite element approximation of the solution to the direct problem, a similar estimate for the finite element approximation of the adjoint problem and an estimate corresponding to the approximation of the coefficient. The latter part is characterized by the residual $R_\eps$ in all estimates.

Explicit solution schemes and numerical testing, including the proper choice of regularization parameter $\alpha$, will be the subject of forthcoming papers.

\bibliographystyle{plainnat}
\bibliography{../mybibliography}

\begin{thebibliography}{10}

\bibitem{bks11}
A.B. Bakushinsky, M.Yu. Kokurin, and A.~Smirnova.
\newblock {\em Iterative Methods for Ill-Posed Problems : An Introduction}.
\newblock De Gruyter, Berlin, 2011.

\bibitem{bkr00}
R.~Becker, H.~Kapp, and R.~Rannacher.
\newblock Adaptive finite element methods for optimal control of partial
  differential equations: Basic concept.
\newblock {\em SIAM J. Control Optim.}, 39:113--132, 2000.

\bibitem{b11}
L.~Beilina.
\newblock Adaptive finite element method for a coefficient inverse problem for
  the {M}axwell's system.
\newblock {\em Applicable Analysis}, 90:1461--1479, 2011.

\bibitem{bk10a}
L.~Beilina and M.V. Klibanov.
\newblock A posteriori error estimates for the adaptivity technique for the
  tikhonov functional and global convergence for a coefficient inverse problem.
\newblock {\em Inverse Problems}, 26:045012, 2010.

\bibitem{bk12}
L.~Beilina and M.V. Klibanov.
\newblock {\em Approximate Global Convergence and Adaptivity for Coefficient
  Inverse Problems}.
\newblock Springer, New York, 2012.

\bibitem{bk12b}
L.~Beilina and M.V. Klibanov.
\newblock The philosophy of the approximate global convergence for
  multidimensional coefficient inverse problems.
\newblock {\em Complex Variables and Elliptic Equations}, 57:277--299, 2012.

\bibitem{bkk10}
L.~Beilina, M.V. Klibanov, and M.Yu. Kokurin.
\newblock Adaptivity with relaxation for ill-posed problems and global
  convergence for a coefficient inverse problem.
\newblock {\em Journal of Mathematical Sciences}, 167:279--325, 2010.

\bibitem{btkm14b}
L.~Beilina, Nguyen T.T., M.V. Klibanov, and J.B. Malmberg.
\newblock Globally convergent and adaptive finite element methods in imaging of
  buried objects from experimental backscattering radar measurements.
\newblock {\em Journal of Computational and Applied Mathematics}, 2014.
\newblock Article in press: http://dx.doi.org/10.1016/j.cam.2014.11.055.

\bibitem{btkm14}
L.~Beilina, Nguyen T.T., M.V. Klibanov, and J.B. Malmberg.
\newblock Reconstruction of shapes and refractive indices from backscattering
  experimental data using the adaptivity.
\newblock {\em Inverse Problems}, 30:105007, 2014.

\bibitem{bcs12}
M.~Bellassoued, M.~Cristofol, and E.~Soccorsi.
\newblock Inverse boundary value problem for the dynamical heterogeneous
  {M}axwell's system.
\newblock {\em Inverse Problems}, 28:095009, 2012.

\bibitem{bk81}
A.L. Bukhge{\u\i}m and M.V. Klibanov.
\newblock Uniqueness in the large of a class of multidimensional inverse
  problems.
\newblock {\em Dokl. Akad. Nauk SSSR}, 260(2):269--272, 1981.

\bibitem{eehj96}
K.~Eriksson, D.~Estep, P.~Hansbo, and C.~Johnson.
\newblock {\em Computational Differential Equations}.
\newblock Studentlitteratur, Lund, 1996.

\bibitem{js95}
C.~Johnson and A.~Szepessy.
\newblock Adaptive finite element methods for conservation laws based on a
  posteriori error estimates.
\newblock {\em Comm. Pure Appl. Math.}, 48:199--234, 1995.

\bibitem{k86}
M.V. Klibanov.
\newblock Uniqueness of the solution of two inverse problems for a {M}axwell
  system.
\newblock {\em Zh. Vychisl. Mat. i Mat. Fiz.}, 26(7):1063--1071, 1119, 1986.

\bibitem{kl10}
K.~Kraft and S.~Larsson.
\newblock The dual weighted residuals approach to optimal control of ordinary
  differential equations.
\newblock {\em BIT}, 50(3):587--607, 2010.

\bibitem{l85}
O.A. Ladyzhenskaya.
\newblock {\em The boundary value problems of mathematical physics}, volume~49
  of {\em Applied Mathematical Sciences}.
\newblock Springer-Verlag, New York, 1985.
\newblock Translated from the Russian by Jack Lohwater [Arthur J. Lohwater].

\bibitem{li05}
S.~Li.
\newblock An inverse problem for {M}axwell's equations in bi-isotropic media.
\newblock {\em {SIAM} J. Math. Anal.}, 37:1027--1043, 2005.

\bibitem{ly07}
S.~Li and M.~Yamamoto.
\newblock An inverse problem for {M}axwell's equations in anisotropic media.
\newblock {\em Chinese Annals of Mathematics, Series B}, 28(1):35--54, 2007.

\bibitem{m14}
J.B. Malmberg.
\newblock {\em A posteriori error estimate in the Lagrangian setting for an
  inverse problem based on a new formulation of {M}axwell's system}, volume 120
  of {\em Springer Proceedings in Mathematics and Statistics}, pages 42--53.
\newblock Springer, 2015.

\bibitem{tgsy95}
A.N. Tikhonov, A.V. Goncharsky, V.V. Stepanov, and A.G. Yagola.
\newblock {\em Numerical Methods for the Solution of Ill-Posed Problems}.
\newblock Kluwer Academic Publishers, Dordrecht, 1995.

\end{thebibliography}

\end{document}